\documentclass[a4paper,12pt]{amsart}
\usepackage{amssymb}
\usepackage{url}
\usepackage[T1]{fontenc}

\input xypic
\xyoption{all}
\xyoption{arc}

\newtheorem{theo}{Theorem}[section]

\newtheorem{prop}[theo]{Proposition}

\newtheorem{lem}[theo]{Lemma}

\newtheorem{coro}[theo]{Corollary}

\def\equat{\refstepcounter{theo}$$~}
\def\endequat{\leqno{\boldsymbol{(\arabic{section}.\arabic{theo})}}~$$}

    \def\FM{{\mathbb{F}}}

    \def\NM{{\mathbb{N}}}

    \def\ZM{{\mathbb{Z}}}

  \def\ab{{\mathbf a}}  
    \def\BC{{\mathcal{B}}}
    \def\CC{{\mathcal{C}}}
    \def\DC{{\mathcal{D}}}

    \def\HC{{\mathcal{H}}}

    \def\LC{{\mathcal{L}}}

    \def\OC{{\mathcal{O}}}

    \def\RC{{\mathcal{R}}}

\def\a{\alpha}

\def\g{\gamma}
\def\G{\Gamma}

\def\D{\Delta}

\def\ph{\varphi}
\def\l{\lambda}

\def\o{\omega}

\def\s{\sigma}

\def\t{\tau}

\DeclareMathOperator{\Ker}{{\mathrm{Ker}}}

\DeclareMathOperator{\Tr}{{\mathrm{Tr}}}

\def\to{\rightarrow}
\def\longto{\longrightarrow}

\def\fonction#1#2#3#4#5{\begin{array}{rccc}
{#1} : & {#2} & \longto & {#3} \\
& {#4} & \longmapsto & {#5} 
\end{array}}

\def\DS{\displaystyle}
\def\SS{\scriptstyle}
\def\SSS{\scriptscriptstyle}

\def\finl{~$\SS \square$}

\def\lexp#1#2{\kern\scriptspace\vphantom{#2}^{#1}\kern-\scriptspace#2}
\def\le{\hspace{0.1em}\mathop{\leqslant}\nolimits\hspace{0.1em}}

\mathchardef\inferieur="321E
\mathchardef\superieur="321F

\def\eqna{\begin{eqnarray*}}
\def\endeqna{\end{eqnarray*}}

\def\itemth#1{\item[${\mathrm{(#1)}}$]}

\def\gfp{{\FM_{\! p}}}

\catcode`\@=11
\long\def\@car#1#2\@nil{#1}
\long\def\@first#1#2{#1}
\long\def\@second#1#2{#2}
\long\def\ifempty#1{\expandafter\ifx\@car#1@\@nil @\@empty
  \expandafter\@first\else\expandafter\@second\fi}
\catcode`\@=12

\newcounter{soussection}[section]

\def\soussection#1{\refstepcounter{soussection}
    \noindent{\bf \arabic{section}.\Alph{soussection}. #1.}
    \addcontentsline{toc}{section}{\quad 
        {\arabic{section}.\Alph{soussection}. #1.}}} 

\DeclareMathOperator{\val}{{\mathrm{val}}}
\DeclareMathOperator{\can}{{\mathrm{can}}}
\DeclareMathOperator{\Br}{{\mathrm{Br}}}
\DeclareMathOperator{\br}{{\mathrm{br}}}

\def\smallgroup{{\SSS{G}}}

\begin{document}

\baselineskip=16pt

\title{Automorphisms of Coxeter groups and Lusztig's \\
conjectures for Hecke algebras \\
with unequal parameters}

\author{C\'edric Bonnaf\'e}
\address{\noindent 
Labo. de Math. de Besan\c{c}on (CNRS: UMR 6623), 
Universit\'e de Franche-Comt\'e, 16 Route de Gray, 25030 Besan\c{c}on
Cedex, France} 

\makeatletter
\email{cedric.bonnafe@univ-fcomte.fr}

\makeatother


\thanks{The author is partly supported by the ANR (Project 
No JC07-192339)}

\date{\today}

\begin{abstract} 
Let $(W,S)$ be a Coxeter system, let $G$ be a finite solvable group 
of automorphisms of $(W,S)$ and let $\ph$ be a weight function 
which is invariant under $G$. Let $\ph_G$ denote the weight function 
on $W^G$ obtained by restriction from $\ph$. The aim of this paper 
is to compare the $\ab$-function, the set of Duflo involutions and 
the Kazhdan-Lusztig cells associated with $(W,\ph)$ 
and with $(W^G,\ph_G)$, provided that Lusztig's Conjectures hold. 
\end{abstract}

\maketitle

\pagestyle{myheadings}

\markboth{\sc C. Bonnaf\'e}{\sc Automorphisms of Coxeter groups and 
Lusztig's conjectures}

\bigskip

\tableofcontents

\bigskip

Let $(W,S)$ be a Coxeter system, with $S$ finite, let $\G$ be a totally 
ordered abelian group and let $\ph : W \to \G$ be a {\it weight function} 
such that $\ph(s) > 0$ for all $s \in S$.

Let $G$ be a group of automorphisms of $W$ stabilizing $S$ and $\ph$. 
We denote by $\ph_G$ the restriction of $\ph$ to the fixed points 
subgroup $W^G$. If $\o \in S/G$ (the orbit set) is such that $W_\o$ 
($=\langle \o\rangle$) is finite, we denote by $s_\o$ the longest element 
of the standard parabolic subgroup $W_\o$ and we set 
$S_G=\{s_\o~|~\o \in S/G$ and $W_\o$ is finite$\}$. 
Then it is well-known that $(W^G,S_G)$ is a Coxeter system and 
that $\ph_G : W^G \to \G$ is a weight function (such that 
$\ph_G(s_\o) > 0$ for all $\o \in S/G$).

With the datum $(W,S,\G,\ph)$ are associated a Hecke algebra 
$\HC(W,S,\G,\ph)$ over the ring $\ZM[\G]$, 
a Kazhdan-Lusztig basis $(C_w)_{w \in W}$ 
of $\HC(W,S,\G,\ph)$, equivalence relations $\sim_\LC$, $\sim_\RC$ and 
$\sim_{\LC\RC}$ and two functions $\ab : W \to \G$ and $\D : W \to \G$ 
(see \cite{lusztig}). We set 
$\DC=\{w \in W~|~\ab(w) =\D(w)\}$. With the datum $(W^G,S_G,\G,\ph_G)$, 
we associate similarly 
$\sim_\LC^\smallgroup$, $\sim_\RC^\smallgroup$, 
$\sim_{\LC\RC}^\smallgroup$, $\ab_G$, 
$\D_G$ and $\DC_G$. The main result of this paper is the following:

\bigskip

\noindent{\bf Theorem A.} 
{\it Assume that $G$ is a finite solvable group and that 
Lusztig's conjectures $(P_1)$, $(P_2)$, $(P_3)$, $(P_4)$ 
in \cite[Chapter 14]{lusztig} hold for the datum 
$(W^{H},S_{H},\G,\ph_{H})$ 
for all subgroups $H$ of $G$. Let $x$, $y \in W^G$. Then:
\begin{itemize}
\itemth{a} $\ab_G(x)=\ab(x)$.

\itemth{b} $\DC_G = \DC \cap W^G$.

\itemth{c} Assume moreover that Lusztig's Conjecture $(P_{13})$ in 
\cite[Chapter 14]{lusztig} hold for the datum $(W^{H},S_{H},\G,\ph_{H})$ 
for all subgroups $H$ of $G$. If $? \in \{\LC,\RC\}$, then $x \sim_? y$ if 
and only if $x \sim_?^\smallgroup y$.

\itemth{d} Assume moreover that Lusztig's Conjectures $(P_9)$ and $(P_{13})$ in 
\cite[Chapter 14]{lusztig} hold for the datum $(W^{H},S_{H},\G,\ph_{H})$ 
for all subgroups $H$ of $G$. Then $x \sim_{\LC\RC} y$ if 
and only if $x \sim_{\LC\RC}^\smallgroup y$.
\end{itemize}}

\bigskip

\noindent{\sc Remark - } If $G$ is not solvable and if we assume moreover 
that Lusztig's conjecture $(P_{12})$ in \cite[Chapter 14]{lusztig} holds, 
then the statements (a), (b) and (c) of Theorem A hold. It is probable 
that (d) also holds, but the proof should rely on a really different argument 
than the one presented here. Indeed, using $(P_{12})$ and 
a theorem of Meinolf Geck \cite{geck induction}, one can reduce the problem 
to the case where $W_\o$ is finite for all $G$-orbits $\o$ in $S$. 
Then, since the automorphism groups of irreducible finite Coxeter systems are solvable, 
one can assume that $G$ is solvable and apply Theorem A above.\finl

\bigskip

The proof of this Theorem makes essential use of reduction modulo $p$. 
Indeed, an easy induction argument reduces immediately the problem 
to the case where $G$ is a $p$-group for some prime number $p$. 
The main ingredient is then the following: the natural stupid map 
$\HC(W^G,S_G,\G,\ph_G) \to \HC(W,S,\G,\ph)^G$ 
is not a morphism of algebras in general. 
However, if we denote by $\Br_G(\HC(W,S,\G,\ph))$ the quotient 
of $\HC(W,S,\G,\ph)^G$ by the two-sided ideal 
$\sum_{H < G} \Tr_{H}^G(\HC(W,S,\G,\ph)^{H})$ 
({\it Brauer's quotient}, see for instance \cite[Page 91]{thevenaz}), then:

\bigskip

\noindent{\bf Proposition B.}
{\it Assume that $G$ is a finite $p$-group. Then the natural linear map 
$\HC(W^G,S_G,\G,\ph_G) \to \Br_G(\HC(W,S,\G,\ph)^G)$ 
is a morphism of algebras whose kernel is generated by $p$. 
Moreover, it preserves the Kazhdan-Lusztig basis.}

\bigskip

\noindent{\bf Acknowledgements.} Part of this work was done while the 
author stayed at the MSRI during the winter 2008. The author wishes to 
thank the Institute for its hospitality and the organizers of the two 
programs for their invitation.

\bigskip

\section{The set-up}

\medskip

\soussection{The group ${\boldsymbol{(W,S)}}$} 
Let $(W,S)$ be a Coxeter system (with $S$ finite), let $\ell : W \to \NM$ 
denote the length function, let $\G$ be a totally ordered abelian 
group and let $\ph : W \to \G$ be a {\it weight function} 
\cite[\S 3.1]{lusztig} that is, a map such that $\ph(ww')=\ph(w)+\ph(w')$ 
whenever $\ell(ww')=\ell(w)+\ell(w')$. 

Let $A$ be the group algebra $\ZM[\G]$: we will use an exponential notation 
for $A$, namely $A=\DS{\mathop{\oplus}_{\g \in \G}} \ZM e^\g$, where 
$e^\g \cdot e^{\g'}=e^{\g+\g'}$ for all $\g$, $\g' \in \G$. 
If $a=\sum_{\g \in \G} a_\g e^\g \in A$, 
we denote by $\deg a$ (resp. $\val a$) the {\it degree} 
(resp. the {\it valuation}) of $a$, that is, the element $\g$ of $\G$ 
such that $a_\g\neq 0$ and which is maximal (resp. minimal) for this 
condition (by convention, $\deg 0 = - \infty$ and $\val 0 = + \infty$).

We shall 
denote by $\HC$ the Hecke algebra associated with the datum $(W,S,\G,\ph)$. 
It is a free $A$-module, with standard basis $(T_w)_{w \in W}$, and 
the multiplication is entirely determined by the following rules:
$$\begin{cases}
T_w T_{w'} = T_{ww'} & \text{if $\ell(ww')=\ell(w)+\ell(w')$;}\\
(T_s - e^{\ph(s)})(T_s+e^{-\ph(s)})=0 & \text{if $s \in S$.}
\end{cases}$$
Note that this implies that $T_w$ is invertible in $\HC$ for all 
$w \in W$. 
This algebra is endowed with an $A$-anti-linear involution $\bar{~} : \HC \to \HC$ 
which is determined by the following properties:
$$\begin{cases}
\overline{e^\g}=e^{-\g} & \text{if $\g \in \G$,}\\
\overline{T}_w=T_{w^{-1}}^{-1} & \text{if $w \in W$.}
\end{cases}$$
By \cite[Theorem 5.2]{lusztig}, 
there exists a unique element $C_w \in \HC$ such that
$$\begin{cases}
\overline{C}_w=C_w, & \\
C_w \equiv T_w \mod \HC_{<0},
\end{cases}$$
where $\HC_{<0} = \DS{\mathop{\oplus}_{w \in W}} A_{<0} T_w$, 
and where $A_{<0} = \DS{\mathop{\oplus}_{\g < 0}} \ZM e^\g$.

Let $\t : \HC \to A$ be the unique $A$-linear map such that 
$$\t(T_w)=\begin{cases} 1 & \text{if $w=1$,} \\ 
0 & \text{otherwise.}\end{cases}$$
If $w \in W$, we set 
$$\D(w)=-\deg \t(C_w),$$
and we denote by $n_w$ the coefficient of $e^{-\D(w)}$ 
in $\t(C_w)$. 
Finally, if $x$, $y \in W$, we write
$$C_x C_y = \sum_{z \in W} h_{x,y,z} C_z,$$
where the $h_{x,y,z}$'s are in $A$ and satisfy 
$\overline{h_{x,y,z}}=h_{x,y,z}$. 

\bigskip

\soussection{The group ${\boldsymbol{(W^G,S_G)}}$} 
Let $G$ be a group of automorphisms of $W$ such that, for all 
$\s \in G$, we have 
$$\s(S)=S\quad\text{and}\quad\ph \circ \s = \ph.$$
If $I$ is a subset of $S$, we denote by $W_I$ the (standard 
parabolic) subgroup of $W$ generated by $I$. 
If $\o \in S/G$ is such that $W_\o$ is finite, 
we denote by $s_\o$ the longest element of $W_\o$. We denote 
by $S_G$ the set of $s_\o$, where $\o$ runs over the set of 
$G$-orbits in $S$ such that $W_\o$ is finite. Recall the following proposition 
\cite[Corollary 3.5 and Proof of Proposition 3.4]{hee}:

\bigskip

\begin{prop}\label{fixes}
$(W^G,S_G)$ is a Coxeter system. Let $\ell_G : W^G \to \NM$ denote 
the corresponding length function and let $x$, $y \in W^G$. Then 
$\ell(xy)=\ell(x)+\ell(y)$ if and only if $\ell_G(xy)=\ell_G(x)+\ell_G(y)$. 
\end{prop}

\bigskip

Let 
$$\fonction{\ph_G}{W^G}{\G}{w}{\ph(w)}$$
denote the restriction of $\ph$ to $W^G$. Then, by Proposition \ref{fixes}, 
\equat\label{phg weight}
\text{\it $\ph_G$ is a weight function.}
\endequat
Therefore, we can 
define $\HC_G$, $\HC_{G,<0}$, $T_w^\smallgroup$, $C_w^\smallgroup$, $\t_G$, 
$\D_G$, $n_z^\smallgroup$ and $h_{x,y,z}^\smallgroup$ with respect 
to $(W^G,S_G,\G,\ph_G)$ in a similar way as $\HC$, $\HC_{<0}$, $T_w$, 
$C_w$, $\t$, $\D$, $n_z$ and $h_{x,y,z}$ were defined with respect 
to $(W,S,\G,\ph)$.

\bigskip

\section{Brauer quotient}

\medskip

\begin{quotation}
\noindent{\bf Hypothesis and notation.} 
{\it From now on, and until the end of section \ref{lc}, 
we fix a prime number $p$ and we assume that $G$ is a finite $p$-group.}
\end{quotation}

\bigskip

\soussection{Definition} 
For all the facts contained in this subsection, the reader may refer 
to \cite[\S 11]{thevenaz}: even though this reference deals 
only with $\OC$-algebras (where $\OC$ is 
a commutative complete local noetherian $\ZM_p$-algebra) which are 
$\OC$-modules of finite type, the proofs 
can be applied almost word by word to our slightly more general situation. 

Let $R$ be a commutative ring and let $M$ be an $RG$-module. 
If $H$ is a subgroup of $G$, we set
$$\fonction{\Tr_H^G}{M^H}{M^G}{m}{\DS{\sum_{\s \in [G/H]} \s(m)}.}$$
Here, $[G/H]$ denotes a set of representatives classes in $G/H$.
We also define
$$\Tr(M)=\sum_{H < G} \Tr_H^G(M^H).$$
This is an $R$-submodule of $M^G$, containing $pM^G$. 
The {\it Brauer quotient} $\Br_G(M)$ is then defined by
$$\Br_G(M)=M^G/\Tr(M)$$
and we denote by $\br_G : M^G \to \Br_G(M)$ the canonical map. 

\bigskip

\begin{lem}\label{base}
Assume that $pR \neq R$ and that $M$ admits an $R$-basis $\BC$ 
which is permuted by the action of $G$. Then $\Br_G(M)$ is a free 
$R/pR$-module with basis $(\br_G(b))_{b \in \BC^G}$. 
\end{lem}

\bigskip

If $M$ is an $R$-algebra and if $G$ acts on $M$ by 
automorphisms of algebra, then $\Tr(M)$ is a two-sided ideal 
of $M^G$ and so $\Br_G(M)$ is an $R$-algebra. Of course, 
$\br_G$ is a morphism of algebras in this case. We recall 
the following result:

\bigskip

\begin{lem}\label{multiplication}
Assume that $pR\neq R$, 
that $M$ is an $R$-algebra, that $G$ acts on $M$ by automorphisms of 
algebra, that $M$ admits an $R$-basis $\BC$ which is permuted 
by $G$ and let us write $ab=\sum_{c \in \BC} \l_{a,b,c} c$ for $a$, 
$b \in \BC$. If $a$, $b \in \BC^G$, then 
$$\br_G(a)\br_G(b)=\sum_{c \in \BC^G} \pi(\l_{a,b,c}) \br_G(c),$$
where $\pi : R \to R/pR$ is the canonical morphism. 
\end{lem}

\bigskip

\soussection{Applications to Hecke algebras} 
Since $G$ stabilizes $S$ and $\ph$, it also acts on $\HC$ by automorphisms 
of $A$-algebra (by $\s(T_w)=T_{\s(w)}$ for all $w \in W$). 
Moreover, it permutes the standard basis $(T_w)_{w \in W}$, 
so it follows from Lemma \ref{base} that:

\bigskip

\begin{coro}\label{base H}
$(\br_G(T_w))_{w \in W^G}$ is an $\gfp[\G]$-basis of the 
$\gfp[\G]$-algebra $\Br_G(\HC)$.
\end{coro}

\bigskip

Now, let
$$\can_G : \HC_G \longto \Br_G(\HC)$$
be the unique $A$-linear map such that 
$$\can_G(T_w^\smallgroup)=\br_G(T_w)$$
for all $w \in W^G$. The main result of this subsection 
is the following:

\bigskip

\begin{prop}\label{reduction}
The map $\can_G : \HC_G \longto \Br_G(\HC)$ is a surjective morphism 
of $A$-algebras whose kernel is $p \HC_G$.
\end{prop}

\bigskip

\begin{proof}
It follows from Corollary \ref{base H} that $\can_G$ is surjective 
and that $\Ker(\can_G)=p\HC_G$. It remains to show that 
$\can_G$ is a morphism of algebras. First, note that if 
$x$, $y \in W^G$ satisfy $\ell_G(xy)=\ell_G(x)+\ell_G(y)$, 
then $\ell(xy)=\ell(x)+\ell(y)$ (by Proposition \ref{fixes}) and so
\begin{multline*}
\can_G(T_x^\smallgroup T_y^\smallgroup)=\can_G(T_{xy}^\smallgroup)=
\br_G(T_{xy})\\
=\br_G(T_x T_y)=\br_G(T_x)\br_G(T_y)=
\can_G(T_x^\smallgroup)\can_G(T_y^\smallgroup). 
\end{multline*}
So it remains to show that, if $\o$ is a $G$-orbit in $S$ such that 
$W_\o$ is finite, 
then 
$$\br_G((T_{s_\o}-e^{\ph(s_\o)})(T_{s_\o} + e^{-\ph(s_\o)})) = 0.\leqno{(?)}$$
Since $s_\o$ is the longest element of $W_\o$, we have 
\cite[Corollary 12.2]{lusztig}
$$C_{s_\o}=\sum_{w \in W_\o} e^{\ph(w)-\ph(s_\o)} T_w$$
and \cite[Theorem 6.6 (b)]{lusztig}
$$(T_{s_\o}-e^{\ph(s_\o)}) C_{s_\o}=0.$$
But $(W_\o)^G=\{1,s_\o\}$. Since $\ph(w)=\ph(\s(w))$ for all $w \in W_\o$ 
and all $\s \in G$, we have 
$$C_{s_\o} \equiv T_{s_\o} + e^{-\ph(s_\o)} \mod \Tr(\HC).$$
This completes the proof of $(?)$. 
\end{proof}

\bigskip

\begin{coro}\label{iso}
$\gfp \otimes_\ZM \HC_G \simeq \Br_G(\HC)$.
\end{coro}

\bigskip

\begin{coro}\label{tau}
If $h \in \HC_G$ and $h' \in \HC^G$ are such that 
$\can_G(h) = \br_G(h')$, then $\t_G(h) \equiv \t(h') \mod pA$. 
\end{coro}

\bigskip

\begin{prop}\label{C mod p}
If $w \in W^G$, then $\can_G(C_w^\smallgroup) = \br_G(C_w)$.
\end{prop}

\bigskip

\begin{proof}
Let $C=\can_G(C_w^\smallgroup)-\br_G(C_w)$. Then 
$$\overline{C}=C.$$
where $\overline{~\vphantom{a}} : \Br_G(\HC) \to \Br_G(\HC)$ is defined 
by $\overline{\br_G(h)}=\br_G(\overline{h})$ for all $h \in \HC^G$. 
Moreover, there exists a family $(\a_w)_{w \in W^G}$ 
of elements of $\gfp \otimes_\ZM A_{<0}$ such that 
$$C=\sum_{w \in W^G} \a_w \br_G(T_w).$$ 
Assume that $C \neq 0$ and let $w$ be maximal (for the Bruhat order) 
such that $\a_w \neq 0$. Then 
$$\overline{C}=\overline{\a}_w \br_G(T_{w^{-1}}^{-1}) + 
\sum_{\stackrel{x \in W^G}{x \neq w}} \overline{\a}_x \br_G(T_{x^{-1}}^{-1}).$$
Therefore, the coefficient of $\br_G(T_w)$ in $\overline{C}$ is equal 
to $\overline{\a}_w$. But $C=\overline{C}$, so 
$\a_w=\overline{\a}_w$. Since $\a_w \neq 0$ and $\a_w \in 
\gfp \otimes_\ZM A_{<0}$, we get a contradiction. So $C=0$, 
as desired.
\end{proof}

\bigskip

\begin{coro}\label{structure}
If $x$, $y$, $z \in W^G$, then 
$h_{x,y,z} \equiv h_{x,y,z}^\smallgroup \mod pA$ and 
$\t(C_z) \equiv \t_G(C_z^G) \mod pA$.
\end{coro}

\bigskip

\begin{proof}
This follows immediately from Proposition \ref{C mod p}, from 
Lemma \ref{multiplication} and from Corollary \ref{tau}.
\end{proof}

\section{Lusztig's conjectures}\label{lc}

\medskip

\def\pre#1{\hspace{0.1em}\mathop{\leqslant}_{#1}\nolimits\hspace{0.1em}}
\def\preg#1{\hspace{0.1em}\mathop{\leqslant}_{#1}^\smallgroup
\nolimits\hspace{0.1em}}

\soussection{Cells}
With $(W,S,\G,\ph)$ are associated preorder relations 
$\pre{\LC}$, $\pre{\RC}$ and $\pre{\LC\RC}$ on $W$ as defined 
in \cite[\S 8.1]{lusztig}. The associated equivalence relations 
are denoted by $\sim_\LC$, $\sim_\RC$ and $\sim_{\LC\RC}$ respectively. 
The equivalence classes for the relation $\sim_\LC$ 
(respectively $\sim_\RC$, respectively $\sim_{\LC\RC}$) are called 
left (respectively right, respectively two-sided) cells of $W$ 
(or for $(W,S,\G,\ph)$ if it is necessary to emphasize the weight 
function).

Similarly, with $(W^G,S_G,\G,\ph_G)$ are associated preorder relations 
$\preg{\LC}$, $\preg{\RC}$ and $\preg{\LC\RC}$ on $W$. 
The associated equivalence relations 
are denoted by $\sim_\LC^\smallgroup$, $\sim_\RC^\smallgroup$ 
and $\sim_{\LC\RC}^\smallgroup$ respectively. We shall compare 
in this section the (left, right or two-sided) 
cells of $W$ and the ones of $W^G$.

\bigskip

\soussection{Boundedness} 
Following Lusztig \cite[\S 13.2]{lusztig}, 
we say that $(W,S,\G,\ph)$ is {\it bounded} if there exists 
$\g_0 \in \G$ such that $\deg \t(T_x T_y T_z) \le \g_0$ 
for all $x$, $y$ and $z \in W$. Lusztig has conjectured 
\cite[Conjecture 13.4]{lusztig} that $(W,S,\G,\ph)$ is always 
bounded.

\bigskip

\begin{quotation}
{\bf Hypothesis.} 
{\it From now on, and until the end of this paper, we assume that 
$(W,S,\G,\ph)$ and $(W^G,S_G,\G,\ph_G)$ are bounded. Recall that 
$p$ is a prime number and that $G$ is a finite $p$-group.}
\end{quotation}

\bigskip

\noindent{\sc Remark - } A finite group is of course bounded. 
An affine Weyl group is also bounded \cite[\S 13.2]{lusztig}.\finl

\bigskip

By \cite[Lemma 13.5 (b)]{lusztig}, this hypothesis allows us to define 
Lusztig's function $\ab : W \to \G$ by 
$$\ab(z)=\max_{x,y \in W} \bigl(\deg h_{x,y,z}\bigr).$$
If $x$, $y$, $z \in W$, we shall denote by $\g_{x,y,z^{-1}}$ the unique 
element of $\ZM$ such that 
$$h_{x,y,z} \equiv \g_{x,y,z^{-1}} e^{\ab(z)}  
\mod\Bigl(\mathop{\oplus}_{\g < \ab(z)} \ZM e^\g\Bigr).$$
Similarly, we define a function $\ab_G : W^G \to \G$ 
and elements $\g_{x,y,z^{-1}}^\smallgroup$ of $\ZM$ 
(for $x$, $y$, $z \in W^G$).

\bigskip

Let $\DC=\{z \in W~|~\ab(z)=\D(z)\}$. If $I \subseteq S$, we denote by $\ab_I$ 
the analogue of the function $\ab$ but defined for $W_I$ instead of $W$: 
if $z \in W_I$, then 
$$\ab_I(z)=\max_{x,y \in W_I} \deg h_{x,y,z}.$$
%
\bigskip

\noindent{\bf Lusztig's Conjectures for ${\boldsymbol{(W,S,\G,\ph)}}$.} 
{\it With the above notation, we have:

\begin{itemize}
\item[{\bf ${\boldsymbol{P}}_{\boldsymbol{1}}$.}] If $z \in W$, then 
$\ab(z) \le \D(z)$.

\item[{\bf ${\boldsymbol{P}}_{\boldsymbol{2}}$.}] If $d \in \DC$ and if $x$, 
$y \in W$ satisfy $\g_{x,y,d} \not= 0$, then $x=y^{-1}$.

\item[{\bf ${\boldsymbol{P}}_{\boldsymbol{3}}$.}] If $y \in W$, then 
there exists a unique $d \in \DC$ such that $\g_{y^{-1},y,d} \not= 0$.

\item[{\bf ${\boldsymbol{P}}_{\boldsymbol{4}}$.}] If $z' \pre{\LC\RC} z$, 
then $\ab(z) \le \ab(z')$. Therefore, if $z \sim_{\LC\RC} z'$, then 
$\ab(z)=\ab(z')$. 

\item[{\bf ${\boldsymbol{P}}_{\boldsymbol{5}}$.}] If $d \in \DC$ and 
$y \in W$ satisfy $\g_{y^{-1},y,d} \not= 0$, then 
$\g_{y^{-1},y,d}=n_d=\pm 1$.

\item[{\bf ${\boldsymbol{P}}_{\boldsymbol{6}}$.}] If $d \in \DC$, then $d^2=1$.

\item[{\bf ${\boldsymbol{P}}_{\boldsymbol{7}}$.}] If $x$, $y$, $z \in W$, 
then $\g_{x,y,z}=\g_{y,z,x}$. 

\item[{\bf ${\boldsymbol{P}}_{\boldsymbol{8}}$.}] If $x$, $y$, $z \in W$ 
satisfy $\g_{x,y,z} \not= 0$, then $x \sim_\LC y^{-1}$, $y \sim_\LC z^{-1}$ and 
$z \sim_\LC x^{-1}$. 

\item[{\bf ${\boldsymbol{P}}_{\boldsymbol{9}}$.}] If $z' \pre{\LC} z$ and 
$\ab(z')=\ab(z)$, then $z' \sim_\LC z$.

\item[{\bf ${\boldsymbol{P}}_{\boldsymbol{10}}$.}] If $z' \pre{\RC} z$ and 
$\ab(z')=\ab(z)$, then $z' \sim_\RC z$.

\item[{\bf ${\boldsymbol{P}}_{\boldsymbol{11}}$.}] If $z' \pre{\LC\RC} z$ and 
$\ab(z')=\ab(z)$, then $z' \sim_{\LC\RC} z$.

\item[{\bf ${\boldsymbol{P}}_{\boldsymbol{12}}$.}] If $I \subset S$ and 
$z \in W_I$, then $\ab_I(z)=\ab(z)$.

\item[{\bf ${\boldsymbol{P}}_{\boldsymbol{13}}$.}] Every left cell 
$\CC$ of $W$ contains a unique element $d \in \DC$. If 
$y \in \CC$, then $\g_{y^{-1},y,d} \not= 0$.

\item[{\bf ${\boldsymbol{P}}_{\boldsymbol{14}}$.}] If $z \in W$, then 
$z \sim_{\LC\RC} z^{-1}$. 

\item[{\bf ${\boldsymbol{P}}_{\boldsymbol{15}}$.}] If $x$, $x'$, $y$, $w \in W$ 
are such that $\ab(y)=\ab(w)$, then 
$$\sum_{y' \in W} h_{w,x',y'} \otimes_\ZM h_{x,y',y} = 
\sum_{y' \in W} h_{y',x',y} \otimes_\ZM h_{x,w,y'}$$
in $A \otimes_\ZM A$.
\end{itemize}}

\bigskip

Let us recall the following result:

\bigskip

\begin{lem}\label{P}
Assume that Lusztig's Conjectures $(P_1)$, $(P_2)$, $(P_3)$ 
and $(P_4)$ hold for $(W,S,\G,\ph)$. Then:
\begin{itemize}
\itemth{a} Lusztig's Conjectures $(P_5)$, $(P_6)$, $(P_7)$ 
and $(P_8)$ hold for $(W,S,\G,\ph)$.

\itemth{b} If $d \in \DC$, then $\g_{d,d,d}=n_d = \pm 1$.

\itemth{c} If $x \in W$ and if $d \in \DC$ is the unique element 
of $W$ such that $\g_{x^{-1},x,d} \neq 0$, then $\g_{x,d,x^{-1}} = \pm 1$.
\end{itemize}
\end{lem}

\bigskip

\begin{proof}
(a) is proved in \cite[Chapter 14]{lusztig}. 

\medskip

(b) By $(P_6)$, we get that $d^2=1$. By $(P_3)$, there exists a unique 
$e \in \DC$ such that $\g_{d,d,e} \neq 0$. By $(P_5)$, this implies that 
$\g_{d,d,e} = n_e = \pm 1$. By $(P_7)$, this implies that 
$\g_{e,d,d} = \pm 1$. By $(P_2)$, we get that $e=d^{-1}=d$.

\medskip

(c) If $x \in W$ and if $d \in \DC$ is the unique element 
of $W$ such that $\g_{x^{-1},x,d} \neq 0$, then 
$\g_{x,d,x^{-1}} = \g_{x^{-1},x,d} = \pm 1$ by $(P_7)$ and $(P_5)$. 
\end{proof}

\bigskip

We can now state the main result of this paper (from which the Theorem 
A in the introduction follows easily by an induction argument on the order of $G$):

\bigskip

\begin{theo}\label{quasi-split}
Recall that $G$ is a finite $p$-group. 
Assume that Lusztig's conjectures $(P_1)$, $(P_2)$, $(P_3)$ 
and $(P_{4})$ hold for both $(W,S,\G,\ph)$ and $(W^G,S_G,\G,\ph_G)$. 
Let $x$ and $y$ be two elements of $W^G$. Then:
\begin{itemize}
\itemth{a} $\ab_G(x)=\ab(x)$.

\itemth{b} $\DC_G = \DC \cap W^G$ ($=\DC^G$).

\itemth{c} Assume moreover that Lusztig's Conjecture 
$(P_{13})$ holds for both $(W,S,\G,\ph)$ and $(W^G,S_G,\G,\ph_G)$. 
Then $x \sim_\LC^\smallgroup y$ (respectively $x \sim_\RC^\smallgroup y$) 
if and only if $x \sim_\LC y$ (respectively $x \sim_\RC y$). 

\itemth{d} Assume moreover that Lusztig's Conjectures $(P_9)$ and  
$(P_{13})$ hold for both $(W,S,\G,\ph)$ and $(W^G,S_G,\G,\ph_G)$. 
Then $x \sim_{\LC\RC}^\smallgroup y$ 
if and only if $x \sim_{\LC\RC} y$. 
\end{itemize}
\end{theo}

\bigskip

\begin{proof}
(a) By Corollary \ref{structure}, we have, for all $x$, $y$, $z \in W^G$:
\begin{itemize}
\itemth{1} If $\g_{x,y,z^{-1}} \not\equiv 0 \mod p$, then $\ab(z) \le \ab_G(z)$.

\itemth{2} If $\g_{x,y,z^{-1}}^\smallgroup \not\equiv 0 \mod p$, then 
$\ab_G(z) \le \ab(z)$.
\end{itemize}
Now let $z \in W^G$. 
By $(P_3)$, there exists a unique $d\in \DC$ such that 
$\g_{z^{-1},z,d} \neq 0$. From the uniqueness, we get that 
$d \in \DC^G \subseteq W^G$. By Lemma \ref{P} (c), we get that 
$\g_{z,d,z^{-1}} = \pm 1$. So $\ab(z) \le \ab_G(z)$ by (1). 

The same argument shows that 
there exists $d \in \DC_G$ such that $\g_{z,d,z^{-1}}^\smallgroup = \pm 1$, 
so (2) can be applied to get that $\ab_G(z)\le \ab(z)$. The proof of (a) 
is complete.

\medskip

Before going further, let us state the following consequence of (a):
\begin{quotation}
\begin{coro}\label{gamma p}
If $x$, $y$, $z \in W^G$, then 
$\g_{x,y,z} \equiv \g_{x,y,z}^\smallgroup \mod p$.
\end{coro}

\begin{proof}
This follows easily from Theorem \ref{quasi-split} (a) and 
Corollary \ref{structure}.
\end{proof}
\end{quotation}

\bigskip

(b) Let $d \in \DC^G$. By Lemma \ref{P} (b), we have 
$n_d = \pm 1$. Moreover, by Corollary \ref{structure}, we have 
$$\t(C_d) \equiv \t_G(C_d^\smallgroup) \mod pA.$$
This shows that the coefficient of $e^{-\D(d)}$ in $\t_G(C_d^\smallgroup)$ is 
non-zero. So $\D_G(d) \le \D(d)$. But, by $(P_1)$, 
$$\ab_G(d) \le \D_G(d) \le \D(d)=\ab(d).$$
So $\ab_G(d) = \D_G(d) = \D(d)=\ab(d)$ by (a). In particular, 
$d \in \DC_G$.

The same argument shows that, if $d \in \DC_G$, then $\D(d) \le \D_G(d)$ 
and again we get similarly that $d \in \DC$. The proof of (b) is complete.

\bigskip

(c) Let $d$ (respectively $e$) be the unique element 
of $\DC$ such that $\g_{x^{-1},x,d} =\pm 1$ (respectively 
$\g_{y^{-1},y,e} =\pm 1$). By uniqueness, we have $d$, $e \in \DC^G=\DC_G$. 
By Corollary \ref{gamma p}, we also get $\g_{x^{-1},x,d}^\smallgroup \neq 0$ 
and $\g_{y^{-1},y,e}^\smallgroup \neq 0$. Therefore, by $(P_8)$, we have 
$$x \sim_\LC d,\quad x \sim_\LC^\smallgroup d, 
\quad y \sim_\LC e\quad\text{and}\quad y \sim_\LC^\smallgroup e.$$
But, by $(P_{13})$, we have 
$x \sim_\LC y$ (respectively $x \sim_\LC^\smallgroup y$) 
if and only if $d=e$. This proves (c). 

\bigskip

(d) Recall that $(P_9)$ implies $(P_{10})$. Moreover, it follows easily 
from $(P_4)$, $(P_9)$ and $(P_{10})$ that $\sim_{\LC\RC}$ 
(respectively $\sim_{\LC\RC}^\smallgroup$) is the equivalence 
relation generated by $\sim_\LC$ and $\sim_\RC$ (respectively 
$\sim_\LC^\smallgroup$ and $\sim_\RC^\smallgroup$). So (d) 
follows from (c). 
\end{proof}

\bigskip

\soussection{Asymptotic algebra} 
Let $J$ (respectively $J_G$) be the free abelian group 
with basis $(t_w)_{w \in W}$ (respectively $(t_w^\smallgroup)_{w \in W}$). 

\bigskip

\begin{quotation}
\noindent{\bf Hypothesis.} 
{\it In this subsection, and only in this subsection, 
we assume moreover that Lusztig's Conjectures 
$(P_1)$, $(P_2)$,\dots, $(P_{15})$ 
hold for $(W,S,\G,\ph)$ and $(W^G,S_G,\G,\ph_G)$.}
\end{quotation}

\bigskip

By \cite[\S 18.3]{lusztig}, $J$ (respectively $J_G$) can be 
endowed with a structure of associative ring, the multiplication 
being defined by $t_x t_y =\sum_{z \in W} \g_{x,y,z^{-1}} t_z$ 
(respectively $t_x^\smallgroup t_y^\smallgroup =
\sum_{z \in W^G} \g_{x,y,z^{-1}}^\smallgroup t_z^\smallgroup$). 
Then it follows immediately from Corollary \ref{gamma p} and 
from Lemma \ref{multiplication} that:

\bigskip

\begin{theo}\label{asymptotique}
Assume that $G$ is a finite $p$-group and that Lusztig's Conjectures 
$(P_1)$, $(P_2)$,\dots, $(P_{15})$ hold for $(W,S,\G,\ph)$ and  $(W^G,S_G,\G,\ph_G)$. Then 
$$\gfp \otimes_\ZM J_G \simeq \Br_G(J).$$
\end{theo}

\bigskip

\section{Open questions}

\medskip

The results of this paper should be compared with 
\cite[Chapter 14]{lusztig}, where the {\it quasi-split case} is 
considered: more particularly, see 
\cite[Lemmas 16.5, 16.6 and 16.14]{lusztig}. 
This leads to the following questions:

\medskip

\begin{itemize}
\item[$\bullet$] Does Theorem A (d) hold if $G$ is not solvable? It is 
probably the case, but a proof should rely on completely different 
arguments. For the statements (a), (b) and (c), see the remark 
after Theorem A in the introduction.

\item[$\bullet$] Let $z \in W^G$. Is it true that $\D(z) \le \D_G(z)$? 
See \cite[Lemma 16.5]{lusztig} for the quasi-split case.

\item[$\bullet$] Let $x$, $y \in W^G$ be such that $x \preg{\LC} y$. 
Is it true that $x \pre{\LC} y$? See \cite[16.13 (a)]{lusztig} 
for the quasi-split case.
\end{itemize}

\end{document}